\documentclass[preprint,12pt]{elsarticle}

\usepackage{color}
\usepackage{amssymb}
\usepackage{amsthm}
\usepackage{amsmath}
\usepackage{epic}
\usepackage{setspace} 
\usepackage{bm}
\usepackage{multirow} 
\usepackage{booktabs}
\usepackage{graphicx}
\usepackage{subcaption}
\usepackage{threeparttable}
\usepackage{epstopdf}
\usepackage{algorithm}    
\usepackage{algorithmicx} 
\usepackage{algpseudocode}
\newtheorem{theorem}{Theorem}[section]

\newtheorem{lemma}[theorem]{Lemma}

\newtheorem{definition}[theorem]{Definition}

\journal{~~}

\begin{document}
\begin{spacing}{1.15}
\begin{frontmatter}
\title{\textbf{The $k$-core of a graph and its high-order spectra}}

\author[label a]{Chunmeng Liu}\ead{liuchunmeng214@nenu.edu.cn/liuchunmeng0214@126.com}
\author[label b]{Qing Xu\corref{cor}}\ead{qinxu191@163.com}
\author[label b]{Changjiang Bu}\ead{buchangjiang@hrbeu.edu.cn}
\cortext[cor]{Corresponding author}

\address{
\address[label a]{Academy for Advanced Interdisciplinary Studies, Northeast Normal University, Changchun 130024, PR China}
\address[label b]{School of Mathematical Sciences, Harbin Engineering University, Harbin 150001, PR China}
}

\begin{abstract}
The $k$-core of a graph is its largest subgraph with minimum degree at least $k$, a fundamental concept for uncovering hierarchical structures.
In this paper, we establish a connection between the $k$-core and the high-order spectra of graphs, a concept originally introduced by Cvetkovi\'{c}, Doob, and Sachs.
Specifically, we consider the high-order spectra defined via the $k$-adjacency tensor.
Within this framework, we prove that a graph admits a non-empty $k$-core if and only if the spectral radius of the $k$-adjacency tensor is greater than or equal to $1$.
Moreover, when the $k$-core exists, vertices corresponding to positive entries in the Perron vector of the $k$-adjacency tensor belong to the $k$-core.
We thus define the $k$-order eigenvector centrality via the Perron vector, which provides both membership identification and a measure of relative influence within the $k$-core.
Numerical experiments confirm our theoretical findings and illustrate the properties of this centrality measure in some real-world networks.
\end{abstract}

\end{frontmatter}

\section{Introduction}

All graphs discussed in this paper are simple, meaning they are undirected and contain no loops or multiple edges.
The \emph{$k$-core} of a graph $G$ is the maximal subgraph of $G$ in which every vertex has degree at least $k$ within the subgraph.
This concept can be traced back to Erd\H{o}s and Hajnal \cite{Erdos_1966} in 1966 when they studied the chromatic number of a graph.
In 1970, Lick and White \cite{Lick_1970} introduced an equivalent concept known as \emph{degeneracy}, defined as the smallest integer $k$ for which every induced subgraph of $G$ contains a vertex of degree at most $k$.
In 1983, Seidman \cite{Seidman_1983} introduced the $k$-core pruning algorithm, which iteratively removes vertices with degree less than $k$.
Due to its wide applicability, the $k$-core has been used across numerous fields. For a comprehensive overview of its theory and applications, we refer readers to the survey \cite{Kong_2019}.
Much of the existing theoretical work on the $k$-core has focused on its behavior in random graphs (see, e.g., \cite{Luczak_1991,Pittel_1996,Riordan_2008}).
In this paper, our work investigates a deterministic spectral characterization of the $k$-core.
We establish a fundamental link between this core decomposition and the high-order spectra of graphs.

The study of \emph{higher-order spectra} was pioneered, in part, by the work of Cvetkovi\'{c}, Doob, and Sachs, who introduced the concept of quadratic eigenvalues for graphs.
These are defined as complex numbers $\lambda$ for which there exists a nonzero vector $\mathbf{x}$ satisfying a system of quadratic equations derived from the graph's neighborhood structure (see \cite{Cvetkovic_1980}).
This foundational idea was later generalized using the spectral theory of tensors \cite{Liu_2023}.
A \emph{tensor} is a multidimensional array, whose eigenvalues were independently defined by Qi \cite{Qi_2005} and Lim \cite{Lim_2005}.
Tensor's eigenvalues have enabled a wide range of applications, including polynomial optimization \cite{app_3}, network analysis \cite{app_4}, solving multilinear systems \cite{app_5,app_6}, spectral hypergraph theory \cite{app_7,app_8,app_9,app_10,app_11}, etc.
Notably, the quadratic eigenvalues introduced by Cvetkovi\'{c} et al. can be viewed as eigenvalues of specific third-order tensors.
More recently, Liu et al. \cite{Liu_2023} reformulated and extended higher-order spectra through the tensor's eigenvalues, introducing the $k$-adjacency tensor of a graph.
This framework facilitated the generalization of several spectral bounds and graph-theoretic results.
Subsequent work \cite{Liu_2023_2} introduced the $t$-clique tensor, which extended the study of spectral Tur\'{a}n-type extremal problems to the high-order setting.
Further developments in this direction can be found in a growing body of literature \cite{Liu_2024,Liu_2026,Peng_2025,Peng_2026_1,Peng_2026_2,Wang_2025}.

The structural insights provided by the $k$-core and high-order spectra naturally lead to a discussion of vertex centrality-a fundamental tool for quantifying vertex's importance in complex networks.
To compare the importance of different vertices, the vertex's importance needs to be mapped to a real number.
This method of mapping vertex's importance to a real number is known as a vertex centrality measure \cite{Lu_2016,Lu_2025}.
So far, experts and scholars have developed various centrality measures from different perspectives, such as degree centrality, closeness centrality, Katz centrality, eigenvector centrality \cite{Bonacich_1987,Bonacich_2007}, and so on.
Among these, eigenvector centrality is a key centrality measure with widespread applications in areas like social network analysis \cite{Fan_2020} and the Google search engine \cite{Gleich_2015}.
Eigenvector centrality is defined as the positive eigenvector corresponding to the spectral radius of the adjacency matrix.
With the advancement of research into higher-order network structures, scholars have extended centrality measures by leveraging the positive eigenvector associated with the spectral radius of a tensor.
For example, Benson \cite{Benson_2019} developed eigenvector centralities via adjacency tensors for uniform hypergraphs.
Tudisco and Higham \cite{Tudisco_2021} studied a class of tensor's spectral centrality measures for identifying important vertices and hyperedges in hypergraphs.
Tortosa et al. \cite{Tortosa_2021} proposed an algorithm based on the PageRank concept to compute the centralities of vertices in multiplex networks via tensors.
Tudisco et al. \cite{Tudisco_2018} defined eigenvector centrality that relies on the Perron eigenvector of a multihomogeneous map on the three-order tensor representing the multiplex network.
Xu et al. \cite{Xu_2023} extended the notion of eigenvector centrality by defining the $2$-step tensor eigenvector centrality to identify important vertices.
For more tensor-based eigenvector centrality measures, please refer to \cite{Wu_2019,Zhang_2024,Bu_2025}.

In this paper, we first answer two fundamental questions about the $k$-core from a purely spectral perspective:
\begin{enumerate}
    \item Can the \textit{existence} of a $k$-core be determined by the higher-order spectra?
    \item Can the \textit{membership} of vertices in the $k$-core be identified from the higher-order spectra?
\end{enumerate}
We answer both questions affirmatively in Section 3.
Moreover, we propose the \emph{$k$-th order eigenvector centrality}, defined by the Perron vector of the $k$-adjacency tensor.
This measure not only captures a vertex's importance within the $k$-core structure but also generalizes the classical eigenvector centrality.
To compute this centrality, we develop a tensor's spectral algorithm (Algorithm \ref{alg:k_core_centrality}) that determines both the existence of the $k$-core and the importance of vertices within it.

The remainder of this paper is organized as follows.
Section 2 introduces the necessary preliminaries, including definitions and fundamental lemmas.
Our main theoretical results on the connection between the $k$-core and high-order spectra are established in Section 3, along with their detailed proofs.
Building on this, we propose the $k$-th order eigenvector centrality and present an algorithm for its computation in Section 4.
The efficacy of our method is then validated through numerical experiments on real-world networks in Section 5.
Finally, we conclude the paper with a summary and discussion in Section 6.

\section{Preliminaries}

In this section, we begin with the concept of the tensor and its spectra.
Let $\mathbb{C}^{n}$ denote the set of $n$-dimensional vectors over the complex field $\mathbb{C}$.
For a tensor $\mathcal{A} = (a_{i_{1}i_{2}\cdots i_{m}})$ of order $m$ and dimension $n$, and a vector $\mathbf{x} = (x_{1}, \ldots, x_{n})^{\mathrm{T}} \in \mathbb{C}^{n}$, the product $\mathcal{A}\mathbf{x}^{m-1}$ is a vector in $\mathbb{C}^{n}$ whose $i$-th component is given by $\sum_{i_{2},\ldots,i_{m}=1}^{n}a_{ii_{2}\cdots i_{m}}x_{i_{2}}\cdots x_{i_{m}}$.
A complex number $\lambda$ is called an \emph{eigenvalue} of $\mathcal{A}$ if there exists a nonzero vector $\mathbf{x} \in \mathbb{C}^{n}$ such that
\begin{align}\label{equ.1}
\mathcal{A} \mathbf{x}^{m-1} = \lambda \mathbf{x}^{[m-1]},
\end{align}
where $\mathbf{x}^{[m-1]} = (x_1^{m-1}, \ldots, x_n^{m-1})^{\mathrm{T}}$. The vector $\mathbf{x}$ is then called an \emph{eigenvector} of $\mathcal{A}$ corresponding to $\lambda$.
The spectral radius of $\mathcal{A}$ is the largest modulus of its eigenvalues, denoted by $\rho(\mathcal{A})$ (see \cite{Qi_2005}).
The \emph{support} of $\mathbf{x}$ is defined as $\operatorname{supp}(\mathbf{x}) = \{ i \in \{1, 2, \ldots, n\} \mid x_i \neq 0 \}$.

If all entries of a tensor $\mathcal{A}$ are nonnegative, then $\mathcal{A}$ is called a \emph{nonnegative tensor}.
The weak irreducibility of a tensor can be characterized via its associated digraph.
For an order $m$ and dimension $n$ tensor $\mathcal{A} = (a_{i_1 i_2 \cdots i_m})$, define the directed graph $G_{\mathcal{A}} = (V, E)$ with vertex set $V = \{1, 2, \dots, n\}$ and arc set
\[
E = \left\{ (i, j) \mid \exists\, \{i_2, \dots, i_m\} \text{ such that } a_{i i_2 \cdots i_m} \neq 0 \text{ and } j \in \{i_2, \dots, i_m\} \right\}.
\]
The tensor $\mathcal{A}$ is said to be \emph{weakly irreducible} if $G_{\mathcal{A}}$ is strongly connected; otherwise, $\mathcal{A}$ is \emph{weakly reducible} \cite{Perron Th.}.

The following Perron-Frobenius theorem for weakly irreducible nonnegative tensors is essential to our analysis.

\begin{lemma}{\rm \cite{Perron Th.}} \label{thm Perron Th.}
If $\mathcal{A}$ is a nonnegative weakly irreducible tensor, then its spectral radius $\rho(\mathcal{A})$ is the unique positive eigenvalue of $\mathcal{A}$, and there exists a unique positive eigenvector $\mathbf{x}$ (up to a positive scaling factor) corresponding to $\rho(\mathcal{A})$.
\end{lemma}

We now introduce the $k$-adjacency tensor, a generalization of the adjacency matrix of graphs.
For a vertex $i\in V(G)$, the neighborhood of $i$ is the set of all vertices adjacent to $i$, denoted by $N_{G}(i)$.

\begin{definition}\textup{\cite{Liu_2023}}
Let $G$ be a simple graph with $n$ vertices.
The $k$-adjacency tensor $\mathcal{A}^{(k)}(G) = (a_{i_1 i_2 \cdots i_{k+1}})$ is a $(k+1)$-order and $n$-dimension tensor defined by
\[
a_{i_1 i_2 \cdots i_{k+1}} =
\begin{cases}
\dfrac{1}{k!}, & \text{if } \{i_2, \ldots, i_{k+1}\} \subseteq N_G(i_1) \text{ and all indices are distinct}, \\
0, & \text{otherwise}.
\end{cases}
\]
The spectral radius $\rho_k(G)$ is the largest modulus among all eigenvalues of $\mathcal{A}^{(k)}(G)$.
The Perron eigenvector $\mathbf{x}^{(k)}$ is a nonnegative eigenvector corresponding to $\rho_k(G)$.
\end{definition}
When $k=1$, the $1$-adjacency tensor reduces to the standard adjacency matrix of $G$, $\rho_1(G)$ is the usual spectral radius of the graph, and the Perron eigenvector $\mathbf{x}^{(1)}$ corresponds to the classical eigenvector centrality.

\begin{lemma}\textup{\cite{Liu_2023}}\label{lem_1}
Let $\delta(G)$ be the minimum degree of a graph $G$.
The $k$-adjacency tensor $\mathcal{A}^{(k)}(G)$ of $G$ is weakly irreducible if and only if $G$ is connected and $\delta(G)\geq k$.
\end{lemma}

The following lemma establishes a fundamental connection between the spectral radius of the $k$-adjacency tensor and the existence of dense subgraphs, which is crucial for our main results.

\begin{lemma}\textup{\cite{Liu_2023}}\label{lem_2}
Let $G$ be a graph and $k \geq 1$ an integer.
Then $\rho_k(G) \geq 1$ if and only if there exists a subgraph $H \subseteq G$ with $\delta(H) \geq k$.
\end{lemma}

\section{Main results}

In this section, we present our main theoretical results, which establish a spectral characterization of the $k$-core.
First, we provide a spectral criterion for the existence of a $k$-core.
Then, we demonstrate how the Perron eigenvector of the $k$-adjacency tensor localizes the $k$-core.

\begin{theorem}\label{thm:existence}
Let $G$ be a graph and let $k \geq 1$ be an integer. Then $G$ has a $k$-core if and only if $\rho_k(G) \geq 1$.
\end{theorem}

\begin{proof}

\noindent\textbf{($\Rightarrow$)}
Suppose $G$ has a $k$-core $H$.
Then $H$ is a subgraph of $G$ with minimum degree $\delta(H) \geq k$.
By Lemma \ref{lem_2}, the existence of such a subgraph implies $\rho_k(G) \geq 1$.

\medskip

\noindent\textbf{($\Leftarrow$)}
We prove the contrapositive: if $G$ has no $k$-core, then $\rho_k(G) < 1$.

Assume $G$ has no $k$-core.
By definition, this means there exists no subgraph $H^{\prime}$ of $G$ with $\delta(H^{\prime}) \geq k$.
Equivalently, for every subgraph $H^{\prime} \subseteq G$, we have $\delta(H^{\prime}) < k$.
By Lemma \ref{lem_2}, this implies $\rho_k(G) < 1$.
Since the contrapositive holds, the original implication is valid.
\end{proof}

\begin{theorem}\label{thm:localization}
Let $G$ be a graph with a $k$-core, and let $\mathbf{x}^{(k)} = (x_1^{(k)}, \ldots, x_n^{(k)})^T$ be the Perron eigenvector of $\mathcal{A}^{(k)}(G)$.
Let $V_{\text{$k$-core}}$ be the set of vertices in the $k$-core of $G$.
Then $\operatorname{supp}(\mathbf{x}^{(k)})\subseteq V_{\text{$k$-core}}$.
Moreover, if the subgraph induced by $V_{\text{$k$-core}}$ is connected, we have $\operatorname{supp}(\mathbf{x}^{(k)}) = V_{\text{$k$-core}}$.
\end{theorem}

\begin{proof}

Let $\mathbf{x} = \mathbf{x}^{(k)}$ be the Perron eigenvector of $\mathcal{A}^{(k)}(G)$ corresponding to $\rho_k(G) > 0$.
We want to show that $\text{supp}(\mathbf{x}) \subseteq V_{\text{$k$-core}}$, i.e., if $x_v \neq 0$ then $v \in V_{\text{$k$-core}}$.
We prove the contrapositive: if $v \notin V_{\text{$k$-core}}$, then $x_v = 0$.
Let $v \notin V_{\text{$k$-core}}$. By the definition of $k$-core through the peeling process, there exists a sequence of vertex removals:
\[
S_0, S_1, \ldots, S_m
\]
such that
\begin{itemize}
    \item $S_0 = \{ u \in V(G) : d_G(u) < k \}$
    \item $S_{i} = \{ u \in V(G) \setminus \bigcup_{j=0}^{i-1} S_j : d_{G_i}(u) < k \}$ where $G_i$ is the graph after removing $\bigcup_{j=0}^{i-1} S_j$
    \item $v \in S_p$ for some $m \geq p \geq 0$
    \item $V_{\text{$k$-core}} = V(G) \setminus \bigcup_{j=0}^m S_j$
\end{itemize}
We prove by induction on $i$ that for all $u \in S_i$, $x_u = 0$.

\textbf{Base case ($i = 0$):} For any $u \in S_0$, we have $d_G(u) < k$.
By the equation (\ref{equ.1})
\[
\rho_k(G) x_u^k = \sum_{\{u_2, \ldots, u_{k+1}\} \subseteq N_G(u)} x_{u_2} \cdots x_{u_{k+1}}.
\]
Since $d_G(u) < k$, the set $\{u_2, \ldots, u_{k+1}\} \subseteq N_G(u)$ is empty. Therefore:
\[
\rho_k(G) x_u^k = 0.
\]
Since $G$ has a $k$-core, Theorem \ref{thm:existence} implies $\rho_k(G) > 0$ and thus $x_u = 0$.

\textbf{Inductive step:} Assume that for all $u \in \bigcup_{j=0}^{i-1} S_j$, we have $x_u = 0$. Let $v \in S_i$.
Then $d_{G_i}(v) < k$.
Consider the eigenvector equation for $v$ in the original graph $G$,
\[
\rho_k(G) x_v^k = \sum_{\{v_2, \ldots, v_{k+1}\} \subseteq N_G(v)} x_{v_2} \cdots x_{v_{k+1}}.
\]
We analyze the terms in this summation. Each term corresponds to a $k$-element subset of $N_G(v)$.
There are two types of such subsets:

(1). \textbf{Subsets containing at least one vertex from $\bigcup_{j=0}^{i-1} S_j$}.
For these subsets, at least one $x_{v_j} = 0$ (by induction hypothesis), so the product $x_{v_2} \cdots x_{v_{k+1}} = 0$.

(2). \textbf{Subsets entirely contained in $V(G) \setminus \bigcup_{j=0}^{i-1} S_j$}.
These correspond to neighbors of $v$ that remain in $G_i$.
However, since $d_{G_i}(v) < k$, there are fewer than $k$ such neighbors.
Therefore, no $k$-element subset of $N_G(v)$ can be entirely contained in $V(G_i)$.

Hence, all terms in the summation are zero, and we have
\[
\rho_k(G) x_v^k = 0.
\]
Since $\rho_k(G) > 0$, we conclude $x_v = 0$.
By induction, for all $v \in \bigcup_{j=0}^m S_j$ (i.e., for all $v \notin V_{\text{$k$-core}}$), we have $x_v = 0$. Therefore:
\[
\text{supp}(\mathbf{x}) \subseteq V(G) \setminus \left( \bigcup_{j=0}^m S_j \right) = V_{\text{$k$-core}}.
\]
This completes the proof that $\text{supp}(\mathbf{x}^{(k)}) \subseteq V_{\text{$k$-core}}$.

We now prove that $V_{\text{$k$-core}} \subseteq \operatorname{supp}(\mathbf{x}^{(k)})$ under the assumption that the subgraph $H$ induced by $V_{\text{$k$-core}}$ is connected.
That is, we will show $x_v > 0$ for all $v \in V(H)$.
Let $\mathbf{x} = \mathbf{x}^{(k)}$ be the Perron eigenvector of $\mathcal{A}^{(k)}(G)$ corresponding to $\rho_k(G) > 0$. From the first part of the proof, we have $\operatorname{supp}(\mathbf{x}) \subseteq V(H)$.

Consider the $k$-adjacency tensor $\mathcal{A}^{(k)}(H)$ of the induced subgraph $H$.
Since $H$ is the $k$-core, $\delta(H) \geq k$, and by assumption $H$ is connected, Lemma \ref{lem_1} implies that $\mathcal{A}^{(k)}(H)$ is weakly irreducible.
By the Perron-Frobenius theorem (Lemma \ref{thm Perron Th.}), $\mathcal{A}^{(k)}(H)$ has a unique positive eigenvector $\mathbf{y} = (y_v)_{v \in V(H)}$ (up to scaling) corresponding to $\rho_k(H) > 0$.

The key step is to show that the restriction $\mathbf{x}|_{H}$ is a positive eigenvector of $\mathcal{A}^{(k)}(H)$.
For any vertex $v \in V(H)$, the eigenvector equation in $G$ is
\begin{equation}
\rho_k(G) x_v^{k} = \sum_{\{v_2, \ldots, v_{k+1}\} \subseteq N_G(v)} x_{v_2} \cdots x_{v_{k+1}}.
\label{eq:original}
\end{equation}

We analyze the summation in equation (\ref{eq:original}).
Let $T = \{v_2, \ldots, v_{k+1}\}$ be a $k$-element subset of $N_G(v)$.

\begin{itemize}
\item \textbf{Case 1:} If $T \not\subseteq V(H)$, then there exists some $u \in T$ with $u \notin V(H)$. Since $\operatorname{supp}(\mathbf{x}) \subseteq V(H)$, we have $x_u = 0$, and consequently $x_{v_2} \cdots x_{v_{k+1}} = 0$.

\item \textbf{Case 2:} If $T \subseteq V(H)$, then since $H$ is an \emph{induced} subgraph of $G$ and $v \in V(H)$, we have:
\[
T \subseteq N_G(v) \cap V(H) = N_H(v).
\]
\end{itemize}
Hence, all terms in equation (\ref{eq:original}) with $T \not\subseteq V(H)$ vanish, and we obtain
\begin{equation}
\rho_k(G) x_v^{k} = \sum_{\{v_2, \ldots, v_{k+1}\} \subseteq N_H(v)} x_{v_2} \cdots x_{v_{k+1}} \quad \text{for all } v \in V(H).
\label{eq:reduced}
\end{equation}
Equation (\ref{eq:reduced}) is precisely the eigenvector equation for the tensor $\mathcal{A}^{(k)}(H)$:
\[
\lambda y_v^{k} = \sum_{\{v_2, \ldots, v_{k+1}\} \subseteq N_H(v)} y_{v_2} \cdots y_{v_{k+1}},
\]
with $\lambda = \rho_k(G)$ and $y_v = x_v$ for $v \in V(H)$. This confirms that $\mathbf{x}|_{H}$ is an eigenvector of $\mathcal{A}^{(k)}(H)$ with eigenvalue $\rho_k(G)$.
We now establish that $\rho_k(G) = \rho_k(H)$.
Since $H$ is a subgraph of $G$, it follows that $\rho_k(H) \leq \rho_k(G)$.
Moreover, since $\mathbf{x}|_{H}$ is an eigenvector of $\mathcal{A}^{(k)}(H)$ with eigenvalue $\rho_k(G)$, the definition of the spectral radius implies $\rho_k(G) \leq \rho_k(H)$.
Therefore, we conclude that $\rho_k(G) = \rho_k(H)$.

Consequently, $\mathbf{x}|_{H}$ is the Perron eigenvector of $\mathcal{A}^{(k)}(H)$. Since $\mathcal{A}^{(k)}(H)$ is weakly irreducible, its Perron eigenvector is unique and strictly positive. Hence, $x_v > 0$ for every $v \in V(H)$, which completes the proof.
\end{proof}

\section{$k$-th Order Eigenvector Centrality: Definition and Algorithm}

Building upon the spectral characterization of the $k$-core established in Theorems \ref{thm:existence} and \ref{thm:localization}, we now introduce a novel centrality measure that naturally emerges from our theoretical framework.
The Perron vector $\mathbf{x}^{(k)}$ of the $k$-adjacency tensor not only identifies the $k$-core but also provides a measure of vertex importance within it.
This leads us to the following definition.

\begin{definition}
For a graph $G$ and an integer $k \geq 1$, the \emph{$k$-th order eigenvector centrality} of a vertex $i$ is defined as the $i$-th component $x_i^{(k)}$ of the Perron vector of $\mathcal{A}^{(k)}(G)$.
\end{definition}

This centrality measure possesses several properties that make it suitable for analyzing core-periphery structures:

\begin{enumerate}
\item \textbf{Core-Restricted:}
It automatically assigns positive centrality scores exclusively to vertices within the $k$-core, effectively filtering out peripheral vertices.
This provides a principled and automatic focus on the structurally central part of the network relevant to the chosen order $k$.
\item \textbf{Higher-Order Influence:}
This measure assesses vertex importance through participation in higher-order structures like $k$-cliques and other dense $k$-connected subgraphs, capturing influence beyond a vertex's local neighborhood.
\end{enumerate}

Leveraging these properties and our theoretical results, we propose Algorithm \ref{alg:k_core_centrality} for simultaneous $k$-core identification and centrality computation.
This algorithm provides a unified spectral approach to uncovering both the existence and the internal structure of the $k$-core.
Algorithm \ref{alg:k_core_centrality} involves computing the spectral radius and the corresponding eigenvector of a tensor.
Those interested in this topic may refer to \cite{Ng_2010,Liu_2010,Zhou_2013}.

\begin{algorithm}[htbp]
    \caption{$k$-core identification and $k$-th order eigenvector centrality calculation}
    \label{alg:k_core_centrality}
    \begin{algorithmic}
    \setlength{\itemsep}{3pt}
        \State \textbf{Inputs:} A graph $G = (V, E)$, an integer $k \geq 1$
        \State \textbf{Outputs:}
        \State \qquad The $k$-core existence indicator of $G$ and the $k$-th order eigenvector centrality $c(i)$ of vertice $i$ in $V$.

        \State \textbf{Step 1:} Construct the $k$-adjacency tensor $\mathcal{A}^{(k)}$ from the given graph $G$.
        \State \textbf{Step 2:} Choose an initial $n$-dimension vector $\mathbf{x}_{(0)}>0$. Let $\mathcal{B} = \mathcal{A}^{(k)} + \mathcal{I}$ and $\mathbf{y}_{(0)}=\mathcal{B} (\mathbf{x}_{(0)})^{k}$. Set $m=1$.
        \State \textbf{Step 3:} Compute
 
\begin{equation*}
 \begin{split}
 \mathbf{x}_{(m)}&=\frac{\mathbf{y}_{(m-1)}^{[\frac{1}{k}]}}{\parallel \mathbf{y}_{(m-1)}^{[\frac{1}{k}]} \parallel}, \mathbf{y}_{(m)}=\mathcal{B} (\mathbf{x}_{(m)})^{k}\\
 \underline{\lambda}_m&=\min_{x_{(m)_i}>0}\frac{y_{(m)_i}}{(x_{(m)_i})^{k}}, \overline{\lambda}_m=\max_{x_{(m)_i}>0}\frac{y_{(m)_i}}{(x_{(m)_i})^{k}}.
 \end{split}
\end{equation*}

 \State \textbf{Step 4:} If $\underline{\lambda}_m = \overline{\lambda}_m$, the algorithm stops. The spectral radius $\rho(\mathcal{A}^{(k)}) = \underline{\lambda}_m -1$ and $\mathbf{x}_{(m)}$ is the eigenvector corresponding to  $\rho(\mathcal{A}^{(k)})$.
       \State \qquad \qquad \quad If $\rho(\mathcal{A}^{(k)}) \geq 1$, assign to each vertex \(i\) the \(k\)-th order eigenvector centrality \(c(i) = x_{(m)_i}\), where \(c(i) > 0\) indicates that \(i\) belongs to the \(k\)-core.
       \State \qquad \qquad \quad If $\rho(\mathcal{A}^{(k)}) < 1$, output ``No $k$-core exists'', then terminate.

 \State $\hspace{1.5cm}$ Otherwise, replace $m$ by $m+1$ and go to Step 3.

    \end{algorithmic}
\end{algorithm}

\section{Experimental results}

In this section, we conduct a series of experiments to validate our theoretical findings and explore the practical characteristics of the proposed $k$-th order eigenvector centrality.
We begin with a case study on a synthetic graph for verification and illustrative purposes, followed by investigations into the structural correlates and comparative advantages of the measure on real-world networks.

\subsection{Case Study: Verification and Illustration}

We first present a case study on an example graph to verify our theoretical results and demonstrate the behavior of the proposed $k$-th order eigenvector centrality.
For comparison, we adopt three classic centrality measures as benchmarks, namely degree centrality, coreness centrality, and standard eigenvector centrality.
The degree centrality of a vertex $i$ is defined as the number of directly connected neighbors of $i$.
The coreness of a vertex $i$ is the maximum non-negative integer $k$ such that $i$ lies in a $k$-core subgraph where all vertices have a degree of at least $k$.
Eigenvector centrality of a vertex quantifies its importance by assigning a score corresponding to the Perron vector of the network's adjacency matrix.

\begin{figure}[htpb]
  \centering
  \includegraphics[width=70mm]{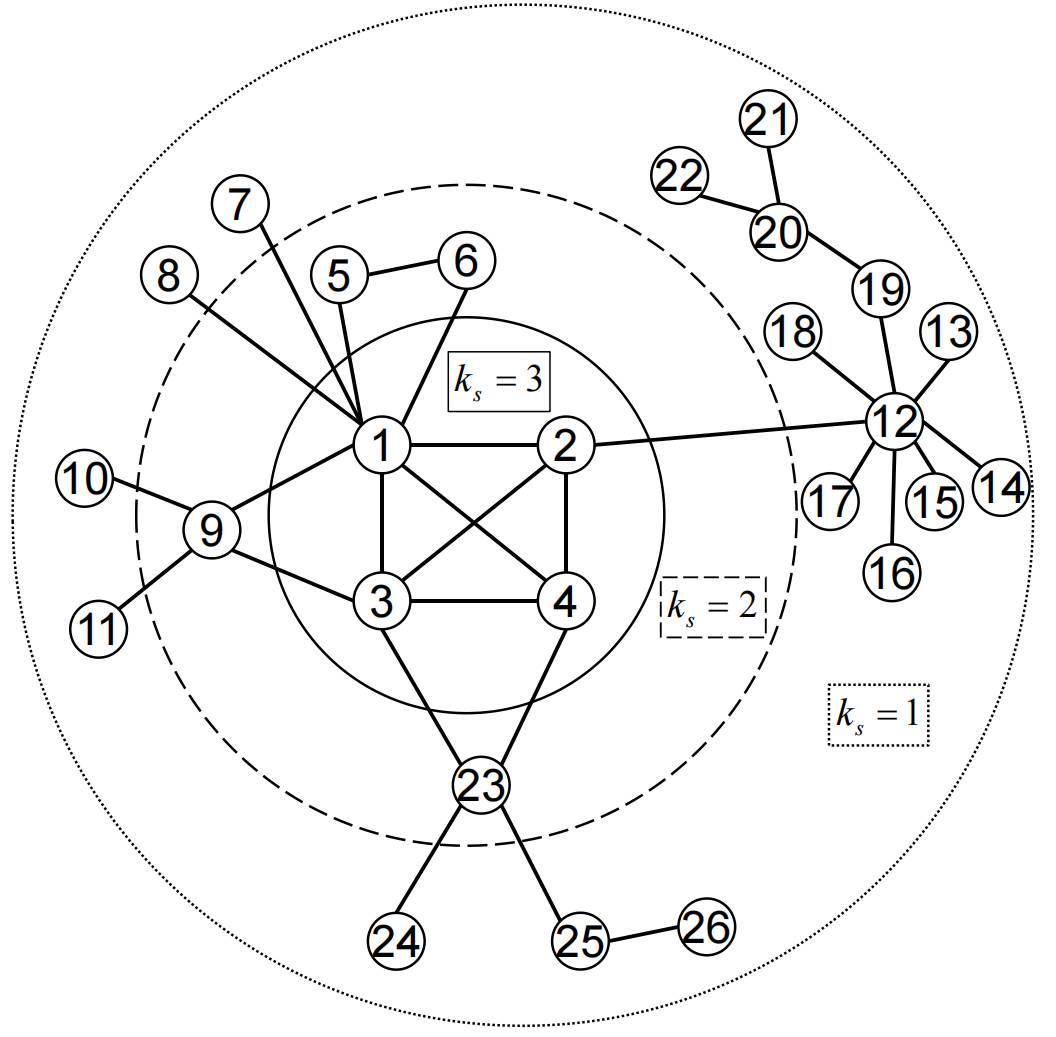}
  \caption{An example graph $G$ with 26 vertices. The labels $k_s = 1, 2, 3$ mark the encircled regions representing the 1-core, 2-core, and 3-core, respectively.}\label{fig2}
\end{figure}

\begin{table}[htbp]
  \centering 
  \caption{Comparison of vertex centrality measures including degree, coresness, $k$-th order eigenvector centrality ($k$ = 1, 2, 3)} 
  \label{tab:example} 
  \vspace{4mm}

  \begin{tabular}{cccccc}
    \toprule 
    \multicolumn{1}{c}{\multirow{3}{*}{Vertex}} &
    \multicolumn{1}{c}{\multirow{3}{*}{Degree}} &
    \multicolumn{1}{c}{\multirow{3}{*}{Coreness}} &
    \multicolumn{3}{c}{$k$-th order eigenvector centrality} \\
    \cmidrule(lr){4-6}
    & & & {$k=1$} & {$k=2$} & {$k=3$} \\
    & & & {$\rho_1(G) = 4.0306$} & {$\rho_2(G) = 4.8132$} & {$\rho_3(G) = 1$} \\
    \midrule

    1 &\textbf{8}&\textbf{3}&\textbf{0.5071}&\textbf{0.5123}&\textbf{0.2500} \\
    2 &4&\textbf{3}&0.3698&0.3871&\textbf{0.250}0 \\
    3 &5&\textbf{3}&0.4369&0.5119&\textbf{0.2500} \\
    4 &4&\textbf{3}&0.3839&0.4482&\textbf{0.2500} \\
    5 &2&2&0.1673&0.1064&0 \\
    6 &2&2&0.1673&0.1064&0 \\
    7 &1&1&0.1258&0&0 \\
    8 &1&1&0.1258&0&0 \\
    9 &4&2&0.2671&0.2334&0 \\
    10 &1&1&0.0663&0&0 \\
    11 &1&1&0.0663&0&0 \\
    12 &\textbf{8}&1&0.1625&0&0 \\
    13 &1&1&0.0403&0&0 \\
    14 &1&1&0.0403&0&0 \\
    15 &1&1&0.0403&0&0 \\
    16 &1&1&0.0403&0&0 \\
    17 &1&1&0.0403&0&0 \\
    18 &1&1&0.0403&0&0 \\
    19 &2&1&0.0434&0&0 \\
    20 &3&1&0.0123&0&0 \\
    21 &1&1&0.0030&0&0 \\
    22 &1&1&0.0030&0&0 \\
    23 &4&2&0.2333&0.2183&0 \\
    24 &1&1&0.0579&0&0 \\
    25 &2&1&0.0617&0&0 \\
    26 &1&1&0.0153&0&0 \\

    \bottomrule 
  \end{tabular}
\end{table}

\begin{figure}[htpb]
  \centering
  \includegraphics[width=0.9\textwidth]{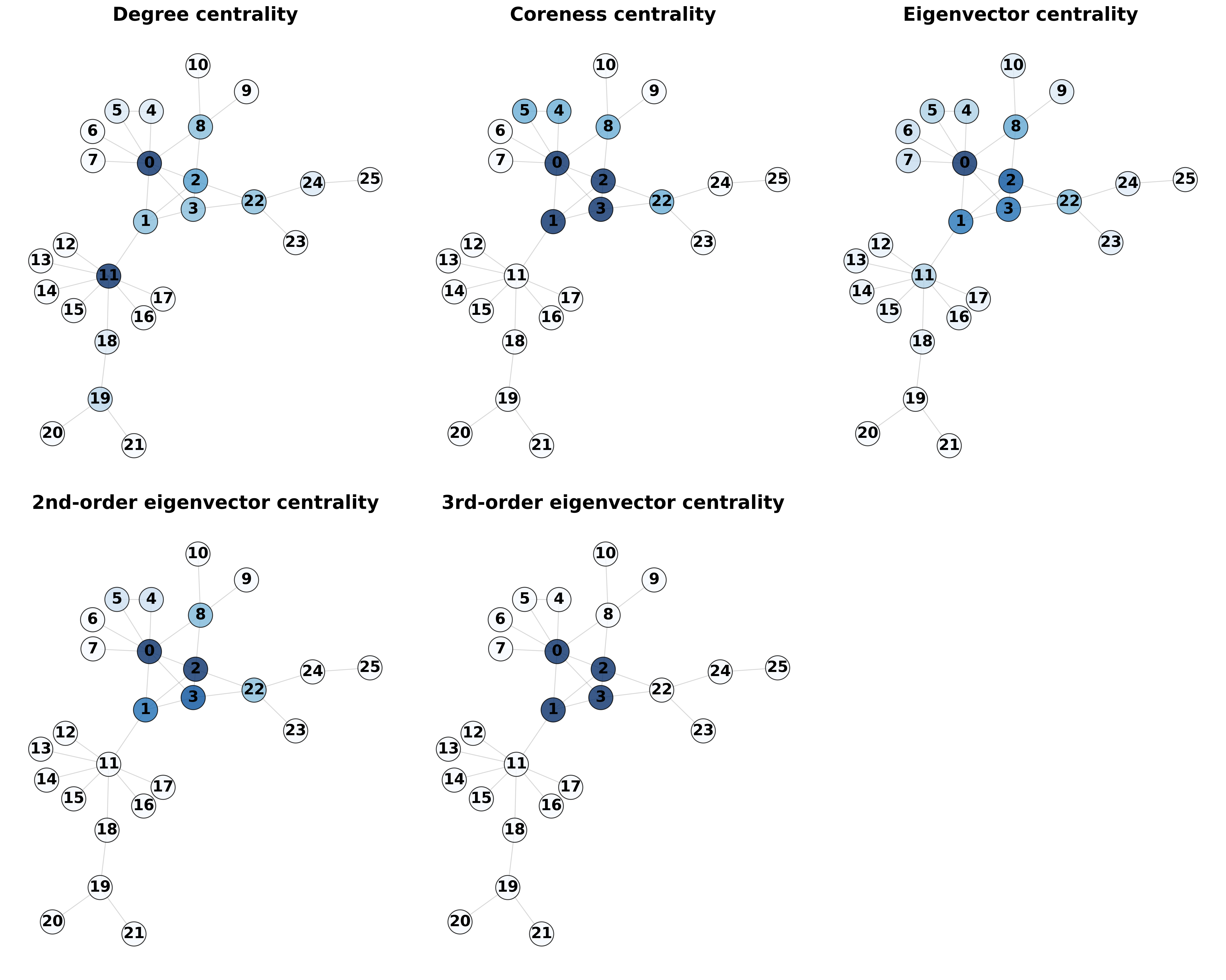}
  \caption{Visualizations of vertices centrality (The darker the color, the higher the centrality)}\label{fig3}
\end{figure}

\textbf{Example:}
Consider a graph $G$ with the 26 vertices shown in Figure \ref{fig2}, which is from Ref. \cite{Kitsak_2010}.
We computed five centrality measures, including degree centrality, coreness, $1$st-order eigenvector centrality, $2$nd-order eigenvector centrality, and $3$rd-order eigenvector centrality.
The results of these computations are presented in Table \ref{tab:example}, with a corresponding visualization of the graph and its centrality distribution provided in Figure \ref{fig3}.

As noted earlier, the $1$st-order eigenvector centrality coincides with the standard eigenvector centrality.
For its $1$-adjacency tensor (i.e., the adjacency matrix), we find that $\rho_1(G)=4.0306 >1$ and the Perron vector is positive entry-wise.
In this vector, vertex $1$ has the highest centrality score.
By Theorems \ref{thm:existence} and \ref{thm:localization}, it follows that $G$ has a $1$-core, which consists of all vertices $1$ to $26$.
This result corresponds to the coreness analysis, where all vertices have a coreness value greater than or equal to 1.

For its $2$-adjacency tensor, we have that $\rho_2(G) = 4.8132$ and the Perron vector is nonnegative entry-wise, with positive entries corresponding to vertices $1$, $2$, $3$, $4$, $5$, $6$, $9$, and $23$.
Consistent with the coreness analysis, these seven vertices constitute the $2$-core of $G$.
This result aligns with Theorems \ref{thm:existence} and \ref{thm:localization}.
In terms of $2$nd-order eigenvector centrality, vertex $1$ still has the highest centrality score.

For its $3$-adjacency tensor, we compute the spectral radius $\rho_3(G) = 1$ and the Perron vector is nonnegative entry-wise, with positive entries corresponding to vertices $1$, $2$, $3$, and $4$.
Only these four vertices have coreness $3$, and they form a $4$-clique.
In this case, they have equal centrality scores.

At last, we consider the $4$-adjacency tensor.
The spectral radius of its $4$-adjacency tensor, $\rho_4(G) = 0$.
By Theorem \ref{thm:existence}, $G$ has no $4$-core, which is consistent with our observation.

\subsection{Structural Correlates of the Centrality Measure}
In this and the following sections, we conduct experiments on five widely used real-world networks: the Zachary's Karate Club network (34 vertices, 78 edges), the USAir97 network (332 vertices, 2,126 edges), the NetScience network (379 vertices, 914 edges), the Email network (1,133 vertices, 5,451 edges), and the Erdos02 network (5,534 vertices, 8,472 edges).
All datasets are publicly available from the SuiteSparse Matrix Collection at https://sparse.tamu.edu/.

The first experiment focuses on centrality scores and the graph's structural characteristics.
When investigating vertex centrality, a fundamental question emerges: which graph's structure influences centrality scores?
For example, degree centrality refers to the number of edges incident to the vertex.
Consider $N_k(i)$, the number of walks of length $k$ starting at vertex $i$ of a non-bipartite connected graph.
Let $s_k(i) = \frac{N_k(i)}{\sum_{j=1}^{n}N_k(j)}$, then for $k \rightarrow \infty$, the vector $(s_k(1),s_k(2),\cdots,s_k(n))^{\mathrm{T}}$ tends towards the eigenvector centrality \cite{Cvetkovic_1997}.
Similarly, the $k$-th order eigenvector centrality of a vertex is likely correlated with its participation in well-connected subgraphs where every vertex has degree at least $k$.
Specifically, we hypothesize that the $2$nd-order eigenvector centrality is strongly correlated with cycle structures.
We demonstrate this correlation by comparing the $2$nd-order eigenvector centrality scores of vertices with the number of cycles containing each vertex across several real-world networks.
Since enumerating all cycles is infeasible for most networks given the enormous computational complexity \cite{Fan_2021}, we count cycles containing the vertex with a length at most 5.

We use $C_k(i)$ $(k=3,4,5)$ to denote the total number of cycles containing the vertex $i$ with a length at most $k$.
Specifically, $C_3(i)$ represents the number of triangles (3-cycle) containing vertex $i$, $C_4(i)$ counts the total number of triangles and 4-cycles containing vertex $i$, and similarly for $C_5(i)$.
We assume $C_k(i)$ $(k=3,4,5)$ be a measure and present the correlation scatter in Figure 3.
As we can see that the $2$nd-order eigenvector centrality exhibits a strong correlation with cycle structures and this correlation strengthens as more cycles are incorporated.
We calculated the Spearman correlation coefficient $r_s$ between the cycle count vector and the 2nd-order eigenvector centrality, with the results reported in Table \ref{tab:coe}.
When we count the number of cycles with length at most 5, the Spearman correlation coefficient are more than 0.9.
This result verify that the $2$nd-order eigenvector centrality is strongly correlated with the cycle structures.

\begin{table}[htbp]
  \centering 
  \caption{Spearman correlation coefficient $r_s$ between the cycles with length at most 3, 4, 5 and 2nd-order eigenvector centrality (2nd EC).} 
  \label{tab:coe} 
  \vspace{4mm}

  \begin{tabular}{cccccc}
    \toprule 
    Networks & $r_s$($C_3$, 2nd EC) & $r_s$($C_4$, 2nd EC) & $r_s$($C_5$, 2nd EC)\\
    \midrule
    Karate & 0.7034 & 0.9394 & 0.9934\\
Jazz & 0.8949 &0.9239 &0.9491\\
USAir97 & 0.9543 &0.9923 &0.9976\\
Email & 0.8586 & 0.9234 & 0.9604\\
Erdos02 & 0.8126 &0.9950 &0.9976\\
    \bottomrule 
  \end{tabular}
\end{table}

\begin{figure}[htbp]
    \centering
    \includegraphics[width=0.8\textwidth]{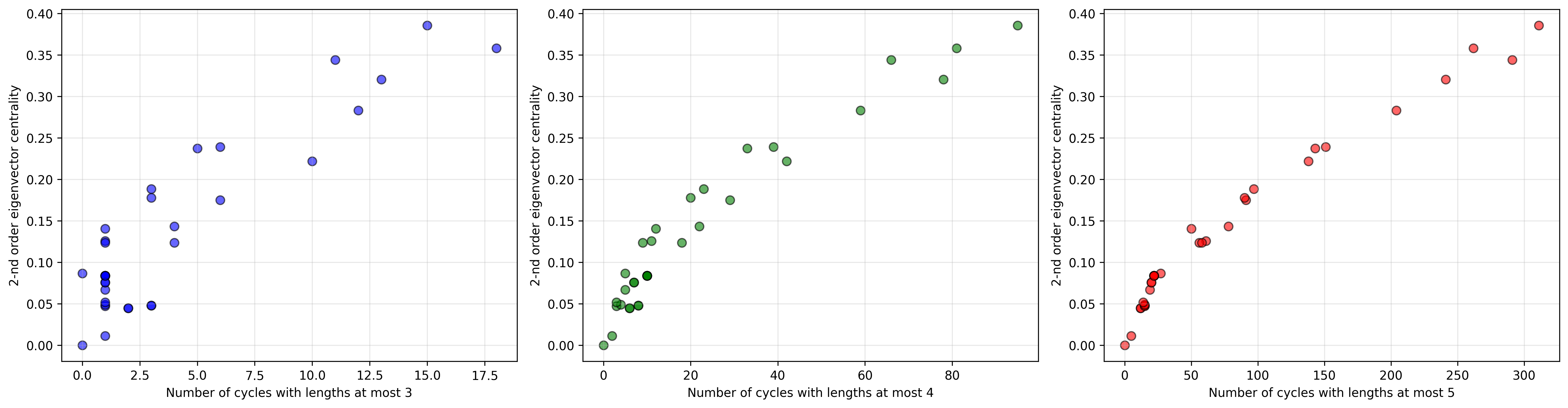}
    \subcaption{Karate network}\label{fig:centrality_cycle_sub1}
    \includegraphics[width=0.8\textwidth]{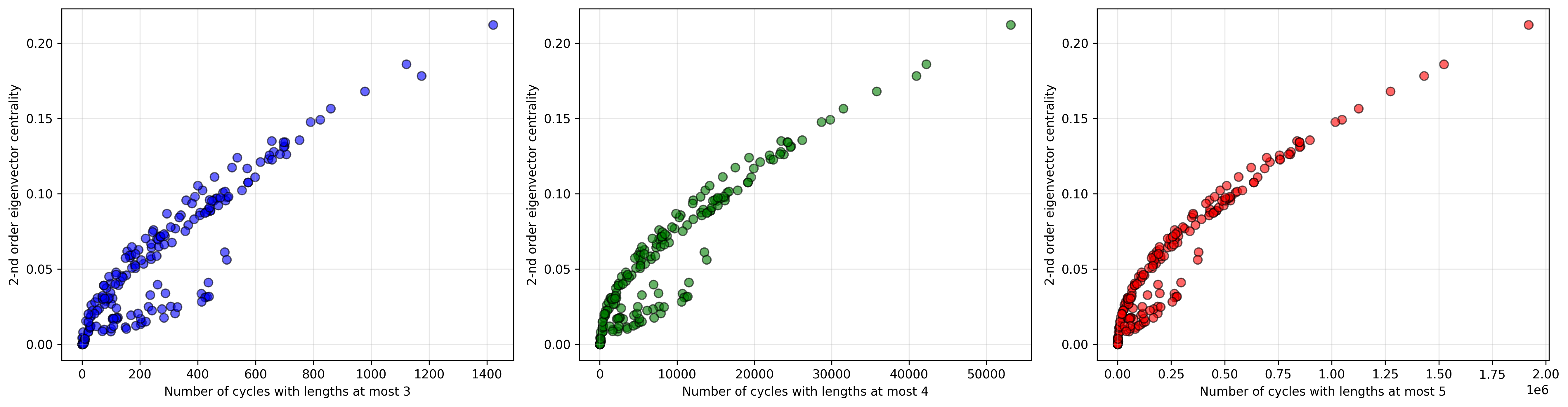}
    \subcaption{Jazz network}\label{fig:centrality_cycle_sub2}
    \includegraphics[width=0.8\textwidth]{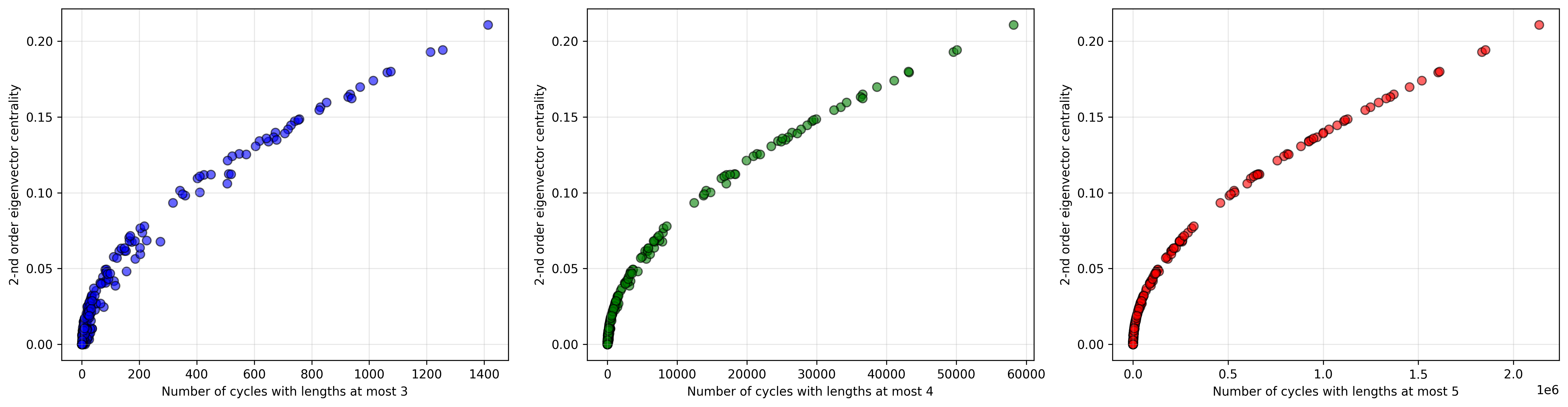}
    \subcaption{USAir97 network}\label{fig:centrality_cycle_sub3}
    \includegraphics[width=0.8\textwidth]{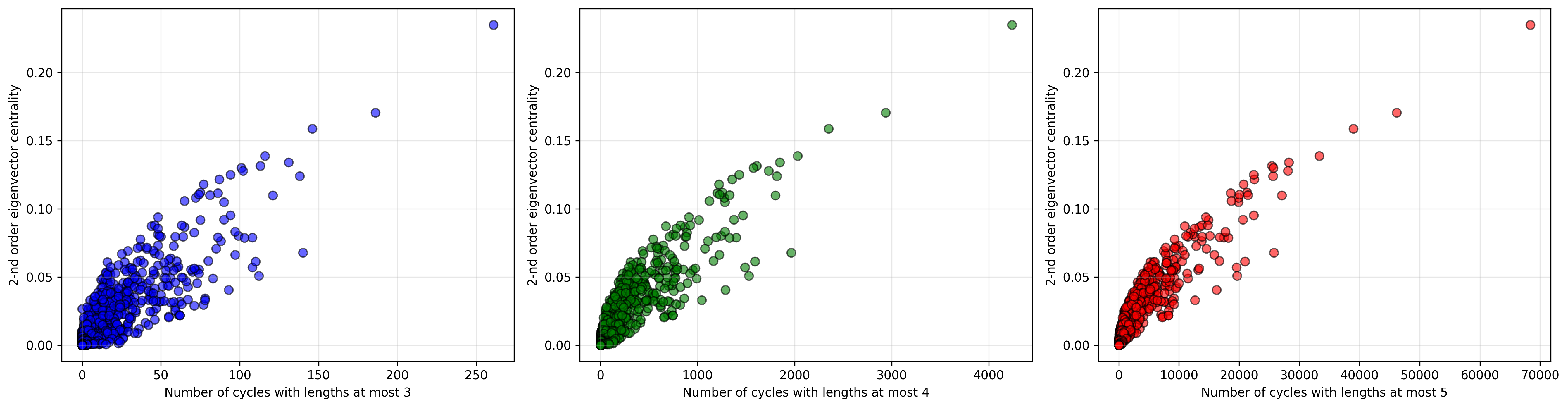}
    \subcaption{Email network}\label{fig:centrality_cycle_sub4}
    \includegraphics[width=0.8\textwidth]{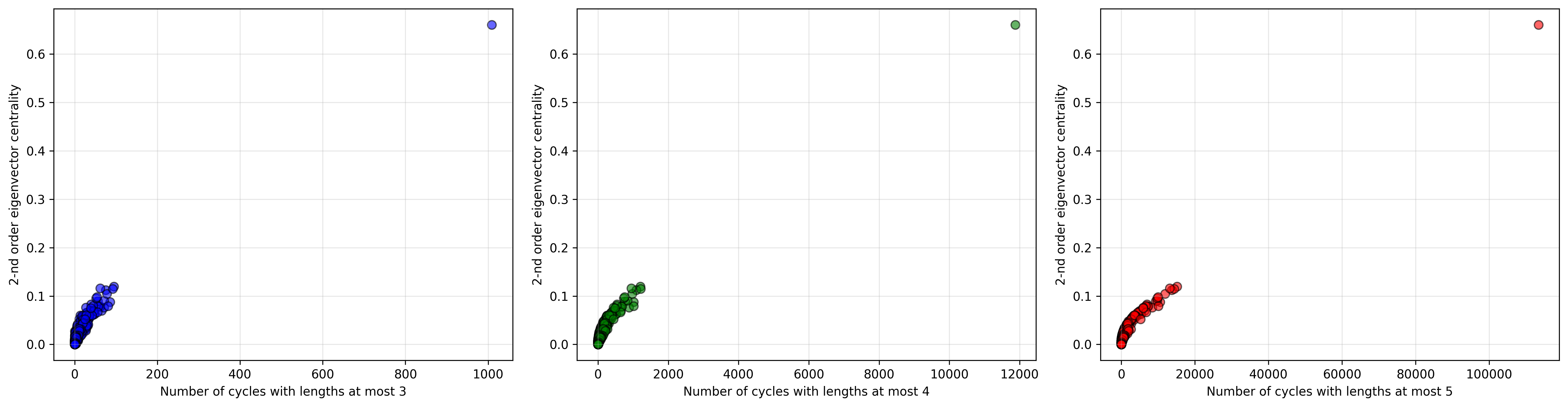}
    \subcaption{Erdos02 network}\label{fig:centrality_cycle_sub5}

    \caption*{Figure 3: Comparison between the number of cycles with lengths up to 3, 4, and 5 and 2nd-order eigenvector centrality across five real-world networks}
    \label{fig:centrality_cycle_comparison} 
\end{figure}

\subsection{Comparative Analysis with Classical Centrality Measures}

To investigate how the proposed 2nd-order eigenvector centrality (2nd EC) relates to classical centrality measures, we compute four centrality indices for each of the five real-world networks in Section 5.2: degree centrality (DC), coreness centrality (CC), standard eigenvector centrality (EC), and the proposed 2nd-order eigenvector centrality (2nd EC).
Pairwise correlation scatter plots between DC, EC, CC and 2nd EC are shown in Fig. 4, where each point represents a vertex.
We also report the Spearman rank correlation coefficients in Table \ref{tab:centraliy_compare}.

\begin{figure}[htbp]
    \centering
    \includegraphics[width=0.8\textwidth]{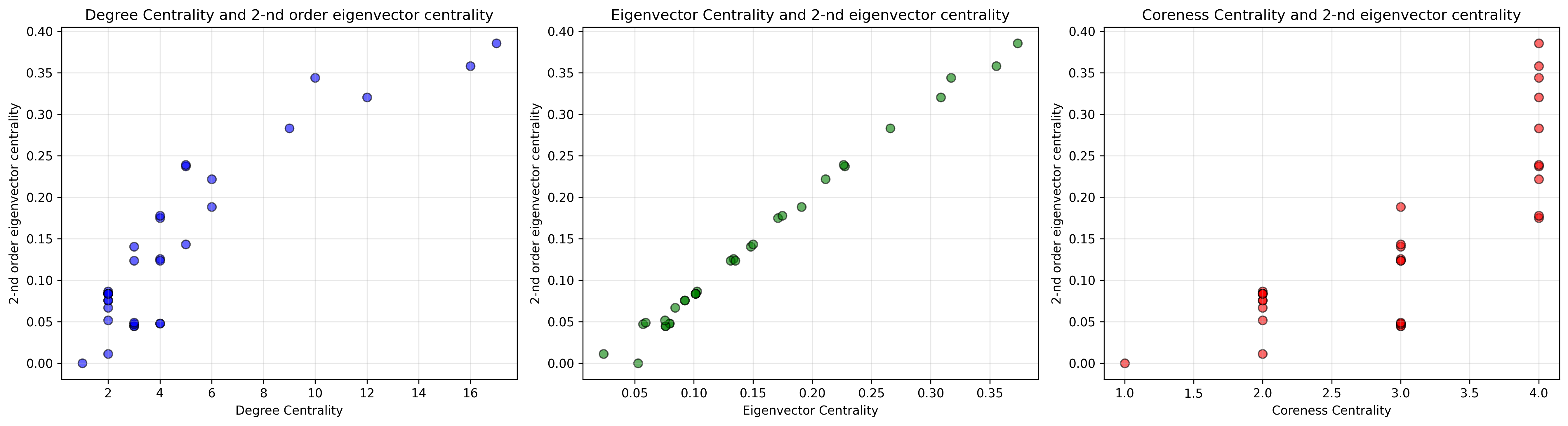}
    \subcaption{Karate network}\label{fig:sub1} 

    \includegraphics[width=0.8\textwidth]{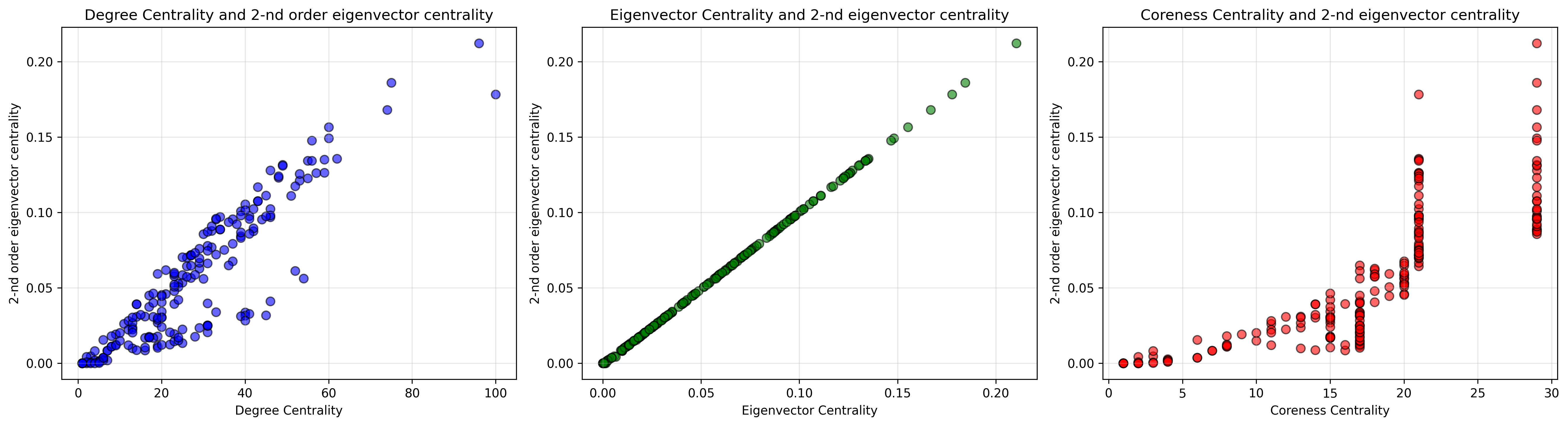}
    \subcaption{Jazz network}\label{fig:sub2}

    \includegraphics[width=0.8\textwidth]{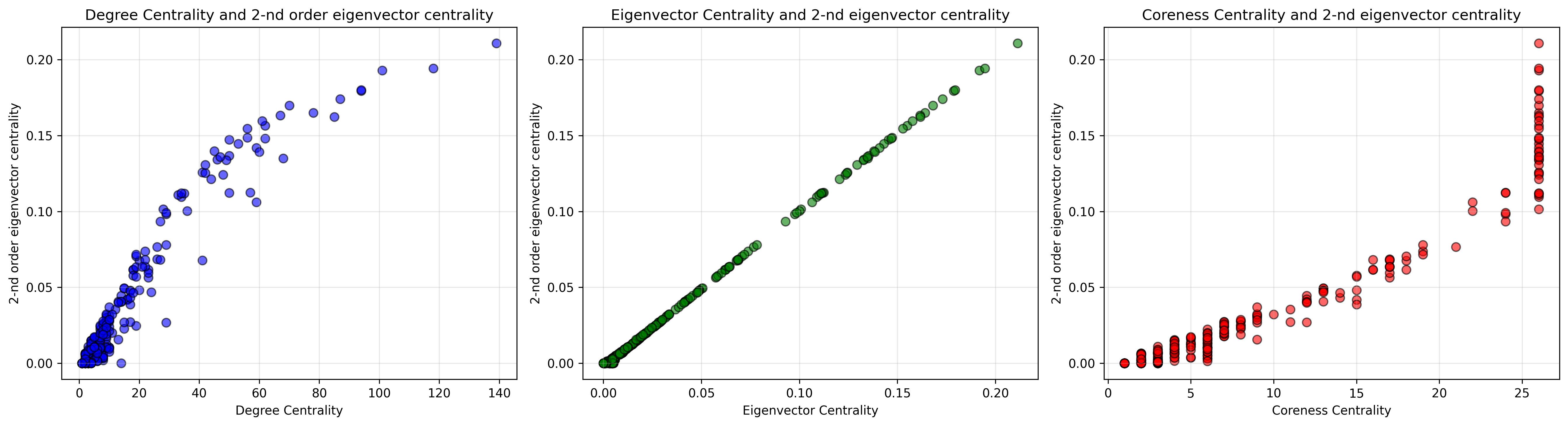}
    \subcaption{USAir97 network}\label{fig:sub3}

    \includegraphics[width=0.8\textwidth]{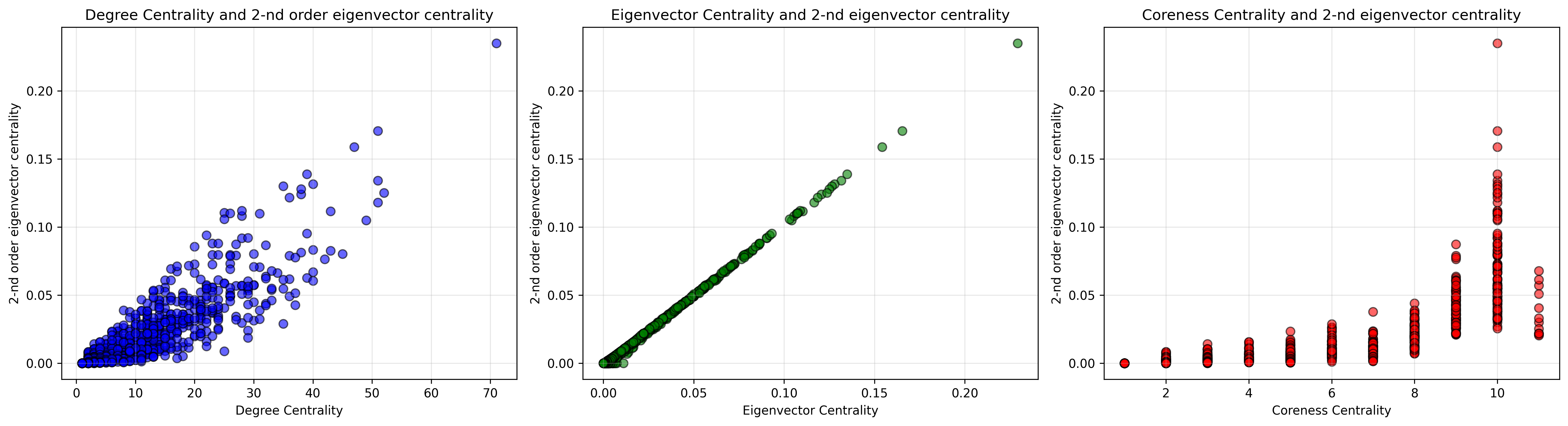}
    \subcaption{Email network}\label{fig:sub4}

    \includegraphics[width=0.8\textwidth]{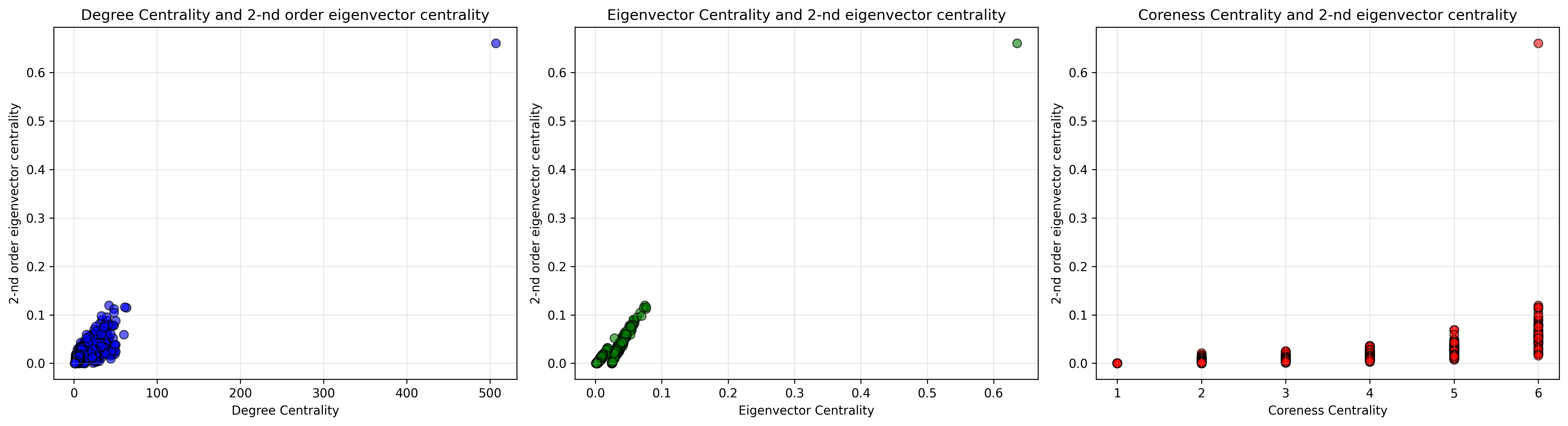}
    \subcaption{Erdos02 network}\label{fig:sub5}

    \caption*{Figure 4: Comparison of Degree centrality, Eigenvector centrality, Coreness and 2nd-order eigenvector centrality on five real-world networks}
    \label{fig:all2}
\end{figure}

\begin{table}[htbp]
  \centering 
  \caption{Spearman correlation coefficient $r_s$ between degree centrality (DC), coreness centrality (CC), standard eigenvector centrality (EC) and 2nd-order eigenvector centrality (2nd EC).} 
  \label{tab:centraliy_compare} 
  \vspace{4mm}

  \begin{tabular}{cccccc}
    \toprule 
    Networks & $r_s$(DC, 2nd EC) & $r_s$(EC, 2nd EC) &$r_s$ (CC, 2nd EC)\\
    \midrule
    Karate & 0.7583 & 0.9873 & 0.7189\\
Jazz & 0.8907 &0.9998 &0.9275\\
USAir97 & 0.9460 &0.9473 &0.9658\\
Email & 0.9763 & 0.9816 & 0.9357\\
Erdos02 & 0.9904 &0.7878 &0.9931\\
    \bottomrule 
  \end{tabular}
\end{table}

For degree and the 2nd-order eigenvector centrality (left column in Figure 4), in small networks (Karate, Jazz), the correlation is strong but not perfect.
As the network's size increase (USAir97 , Email , Erdos02), the relationship becomes increasingly correlated, with Spearman correlation coefficients rising from 0.7583 to 0.9904.

For the eigenvector centrality and the 2nd-order eigenvector centrality (middle column in Figure 4), the two eigenvector-based measures demonstrate high consistency across most networks.
The correlation is nearly perfect in the Jazz network ($r_s= 0.9998$)  and remains strong in the USAir97 ($r_s= 0.9473$) and Email networks ($r_s= 0.9816$).
An exception arises in the Erdos02 collaboration network, where the coefficient drops to 0.7878.
This can be attributed to the fact that 3948 vertices in Erdos02 network belong to the $1$-core but not the $2$-core; by the definition of the 2nd-order eigenvector centrality, these vertices have a centrality score of 0, whereas in standard eigenvector centrality, their scores are non-zero.

The correlation between the coreness centrality (CC) and the 2nd-order eigenvector centrality (2nd EC) exhibits differences, a phenomenon that stems from the distinctions in their definitions: the CC quantifies a vertex's hierarchical position via $k$-core decomposition, while the 2nd EC characterizes a vertex's importance based on 2-adjacency neighborhood.
In the Karate network (\(r_s = 0.7189\)) and the Jazz network (\(r_s = 0.9275\)), the two measures show a moderate correlation.
This indicates that in these two networks, the vertex's core hierarchy reflected by the CC and the 2-adjacency structural importance captured by the 2nd EC exhibit partial alignment but not complete consistency, reflecting differences in the perspectives of the two indicators on vertex's importance.
In the USAir97 network (\(r_s = 0.9658\)) and the Erdos02 network (\(r_s = 0.9931\)), they demonstrate a high correlation.
The correlation in the Email network falls between the above two categories (\(r_s = 0.9357\)), indicating a moderate alignment between the CC and the 2nd EC in this network.

\section{Concluding remarks}

This paper established a novel spectral perspective on the $k$-core decomposition of graphs by investigating its connection to the high-order spectra.
Our main contributions are threefold:

\begin{itemize}
    \item \textbf{Spectral Criterion for Existence:} We proved that a graph admits a non-empty $k$-core \textbf{if and only if} the spectral radius of its $k$-adjacency tensor satisfies $\rho_k(G) \geq 1$.
    \item \textbf{Spectral Localization of the Core:} We showed that the support of the Perron vector of the $k$-adjacency tensor is contained within the $k$-core, and equals it when the core is connected, providing a spectral method to identify its members.
    \item \textbf{A Novel Centrality Measure:} Building on these insights, we proposed the \textbf{$k$-th order eigenvector centrality}, a core-restricted measure that generalizes classical eigenvector centrality and quantifies vertex importance within the $k$-core.
\end{itemize}

Experimental results on real-world networks confirmed our theoretical findings and demonstrated the utility of the proposed centrality.
While our approach offers a principled spectral framework, its computational efficiency for large-scale networks requires further optimization due to the complexity of tensor operations \cite{Lim_2013}.
Future work will focus on developing more scalable algorithms and extending the framework to dynamic network settings.



\end{spacing}
\end{document}